\documentclass[letterpaper, 10 pt, conference]{ieeeconf} 

\IEEEoverridecommandlockouts                              
\usepackage{graphicx}
\usepackage{amsmath,amssymb,amsfonts,eucal}
\usepackage{amscd}
\usepackage{mathrsfs}
\usepackage{bm,amsbsy}
\usepackage[all]{xy}

\newcommand{\newsec}[1]{\medskip {\bf \noindent #1. \quad }}

\makeatletter
\def\@begintheorem#1#2{\medskip\@IEEEtmpitemindent\itemindent\topsep 0pt\itshape\trivlist%
    \item[\hskip \labelsep{\indent\bf #1\ #2.}]\itemindent\@IEEEtmpitemindent}
\def\@endtheorem{\endtrivlist\unskip\medskip}
\makeatother

\newtheorem{proposition}{Proposition}[section]
\newtheorem{lemma}[proposition]{Lemma}

\newtheorem{theorem}[proposition]{Theorem}

\title{\LARGE \bf %
Stokes-Dirac Structures through Reduction of \\ Infinite-Dimensional Dirac Structures}

\author{Joris Vankerschaver, Hiroaki Yoshimura, Melvin Leok, Jerrold E. Marsden
\thanks{\small 
J. Vankerschaver is with the Department of Mathematical Physics and Astronomy, Ghent University, B-9000 Ghent, Belgium. Research supported by by a Postdoctoral Fellowship of the Research Foundation --- Flanders (FWO-Vlaanderen). \newline
H. Yoshimura is with the Department of Applied Mechanics and Aerospace Engineering, Waseda University, Ohkubo, Shinjuku, Tokyo 169-8555, Japan.  Research partially supported by JSPS Grant-in-Aid 20560229 and by JST, CREST.\newline 
M. Leok is with the Department of Mathematics, University of California at San Diego, 9500 Gilman Drive, La Jolla, CA 92093-0112, USA.  Research partially supported by NSF grant DMS-1010687.
\newline 
J. Marsden is with the Control and Dynamical Systems Department, California Institute of Technology, Pasadena, CA 91125, USA.}}

\begin{document}
\maketitle

\begin{abstract}
	We consider the concept of Stokes-Dirac structures in boundary control theory proposed by van der Schaft and Maschke. We introduce Poisson reduction in this context and show how Stokes-Dirac structures can be derived through symmetry reduction from a canonical Dirac structure on the unreduced phase space.  In this way, we recover not only the standard structure matrix of Stokes-Dirac structures, but also the typical non-canonical advection terms in (for instance) the Euler equation.
\end{abstract}

\section{Introduction}

The Hamiltonian formalism has been successfully applied to the description of a variety of field theories, among others the Euler equations for ideal fluids, the Maxwell equations, and the Korteweg-de Vries equation (see \cite{AbMa78, ArKoNe1997, Mo1998} for a historical survey).  Most of these efforts have been directed towards field theories defined on a manifold without boundary, for which a comprehensive theory of Poisson brackets, symmetry reduction and stability analysis exists.  

From the point of view of control and interconnection, it is more natural to consider \emph{open} systems \cite{Willems2007}, and more specifically field theories with varying boundary conditions. As pointed out in \cite{VaMa2002, JeVa2008}, non-zero energy flow through the boundary causes the usual Poisson operator in Hamilton's equation to be no longer skew-symmetric, so that these system fall outside the scope of the traditional Hamiltonian framework.  These difficulties were circumvented in \cite{VaMa2002} by means of a certain notion of infinite-dimensional Dirac structure called a \emph{Stokes-Dirac structure}, whose Hamiltonian equations allow for a non-zero energy flux through the boundary.

Dirac structures \cite{Courant1990} arise naturally in the modeling of a range of mechanical systems (see \cite{DaSc1999} and \cite{MarsdenDirac} for an overview).  In the case of Stokes-Dirac structures, this description is extended by describing the underlying field theory in terms of \emph{differential forms}.  Stokes' theorem then ensures that boundary energy flow is properly incorporated.

While Stokes-Dirac structure have proven to be successful in describing a variety of control-theoretic systems, they are as of yet not fully understood.  The foundational paper \cite{VaMa2002} offers a definition of a Stokes-Dirac structure, but this description often has to be augmented to accommodate specific examples such as Euler's equations for a compressible fluid. 

\newsec{Contributions}
The aims and achievements of this paper are twofold.  In the first part of the paper, we consider field theories whose configuration space $Q$ is a space of forms, on which a symmetry group $G$ acts which is again a space of forms.  We start from the canonical Dirac structure on $T^\ast Q$ and perform symmetry reduction.  We then show that the resulting reduced Dirac structure is precisely a Stokes-Dirac structure in the sense of \cite{VaMa2002}.  We illustrate this procedure by means of the telegrapher's equations, Maxwell's equations, and the equations for a vibrating string.  In the second part of the paper, we then show that a similar reduction philosophy can be used to derive non-canonical Stokes-Dirac structures as well.  We demonstrate this by deriving from first principles the Stokes-Dirac structure for a compressible isentropic fluid.

Throughout this paper, we only treat the case of boundaryless manifolds, a simplification which greatly clarifies the reduction picture.  We return to the issue of boundary conditions at the end of the paper, where we argue how boundary terms can be described by means of extended dual spaces.

\section{Dirac Structures and Reduction}

Dirac structures were described from a mathematical point of view in \cite{Courant1990} and were applied to the modeling of Lagrangian and Hamiltonian mechanical systems in \cite{MarsdenDirac, DaSc1999}.  We recall here the basic concepts and refer to these papers for more information.

\newsec{Dirac Structures on Manifolds}
Let $Q$ be a manifold and define a pairing on $TQ \oplus T^\ast Q$ given by 
\[
	\left<\!\left<(v, \alpha), (w, \beta) \right>\!\right> = 
	\frac{1}{2} (\alpha(w) + \beta(v)).
\]
For a subspace $D$ of $TQ \oplus T^\ast Q$, we define the orthogonal complement $D^\perp$ as the space of all $(v, \alpha)$ such that $\left<\!\left<(v, \alpha), (w, \beta) \right>\!\right> = 0$ for all $(w, \beta)$.  A \emph{\textbf{Dirac structure}} is then a subbundle $D$ of $TQ \oplus T^\ast Q$ which satisfies $D = D^\perp$.

\newsec{The Canonical Dirac Structure}
Let $Q$ be equipped with a symplectic form $\omega$ and note that $\omega$ induces a map $\flat: TQ \rightarrow T^\ast Q$ given by $\flat(v) = \mathbf{i}_v \omega$ for $v \in TQ$.  Since $\omega$ is symplectic, $\flat$ can be inverted and we denote the inverse map by $\sharp: T^\ast Q \rightarrow TQ$, referred to as the \emph{\textbf{Poisson structure}} induced by $\omega$.  It can easily checked that the graph of $\flat$ (or equivalently of $\sharp$), given by 
\begin{align}
	D_{T^\ast Q} & := \{ (v, \flat(v)) : v \in TQ \} 
		\nonumber \\
		& = \{ (\sharp(\alpha), \alpha) : 
			\alpha \in T^\ast Q \}
				\label{poisson}
\end{align}
is a Dirac structure.  In this paper we will exclusively deal with this kind of Dirac structures, but it should be noted that not all Dirac structures can be represented in this way.

\newsec{Reduction of Dirac Structures}
Let $G$ be a Lie group which acts on $Q$ from the right and assume that the quotient space $Q/G$ is again a manifold.  Denote the action of $g \in G$ on $q \in Q$ by $q \cdot g$ and the induced actions of $g \in G$ on $TQ$ and $T^\ast Q$ by $v \cdot g$ and $\alpha \cdot g$, for $v \in TQ$ and $\alpha \in T^\ast Q$.  Note that the action on the cotangent bundle is defined by $\left< \alpha \cdot g, v \right> = \left< \alpha, v \cdot g^{-1} \right>$.  In what follows, we will focus mostly on the reduced cotangent bundle $(T^\ast Q)/G$.  When no confusion is possible, we will denote this space by $T^\ast Q/G$.

Consider now the canonical Dirac structure on $T^\ast Q$.  There are various ways to describe the induced Dirac structure on $T^\ast Q/G$ (see among others \cite{BlSc2001, YoMa2007} for an overview of Dirac reduction theory), but for our purposes, the following point of view is most suitable.  Let $\sharp : T^\ast Q \rightarrow TQ$ be the map \eqref{poisson} used in the definition of $D_{T^\ast Q}$.  The reduced Dirac structure $D_{T^\ast Q/G}$ on $T^\ast Q/G$ can now be described as the graph of a reduced map $[\sharp] : T^\ast (T^\ast Q/G) \rightarrow T(T^\ast Q/G)$ defined as follows.

Let $\pi_G : T^\ast Q \rightarrow T^\ast Q/G$ be the quotient map and consider an element $(\rho, \pi)$ in $T^\ast Q$.  The tangent map of $\pi_G$ at $(\rho, \pi)$ is denoted by $T_{(\rho, \pi)} \pi_G  :  T_{(\rho, \pi)} (T^\ast Q) \rightarrow T_{(\rho, \pi)} (T^\ast Q/G)$, and its dual by $T_{(\rho, \pi)}^\ast \pi_G : T^\ast_{\pi_G(\rho, \pi)} (T^\ast Q/G) \rightarrow T^\ast_{\pi_G(\rho, \pi)} (T^\ast Q)$.  We do not dwell on the exact definition of these maps any further, as the definition will be given in some concrete cases below.  The reduced map $[\sharp]$ now fits into the following commutative diagram:
\begin{equation} \label{prescription}
\xymatrix{
T^\ast_{(\rho, \pi)}(T^\ast Q) \ar[r]^\sharp & T_{(\rho, \pi)}(T^\ast Q) 
	\ar[d]^{T_{(\rho, \pi)} \pi_G} \\
T^\ast_{\pi_G(\rho, \pi)}(T^\ast Q/G) \ar[u]_{T^\ast_{(\rho, \pi)} \pi_G} \ar[r]_{[\sharp]} & T_{\pi_G(\rho, \pi)}(T^\ast Q/G).
}\end{equation}

\medskip
\section{Spaces of Differential Forms}

Throughout this paper, $M$ will be a compact $n$-dimensional manifold without boundary.  In most cases, $M$ can be assumed to be contractible, but this is not necessary.  We denote the space of $k$-forms on $M$ by $\Omega^k$.   

\newsec{The Configuration Space}
The physical problems considered further on will be defined on certain spaces of forms on $M$, and so we single out one specific set of forms, which we denote by $Q := \Omega^k$.   Since $Q$ is a vector space, its tangent bundle can be identified with $TQ = Q \times Q$, while its cotangent bundle is $T^\ast Q = Q \times Q^\ast$.  The latter can be made more explicit by noting that $Q^\ast = \Omega^{n-k}$, where the duality pairing between $Q = \Omega^k$ and $Q^\ast = \Omega^{n-k}$ is given by 
\[
	\left< \rho, \sigma \right> = \int_M \rho \wedge \sigma 
\]
for $\rho \in Q$ and $\sigma \in Q^\ast$.  

For future reference, we note that the tangent bundle $T ( T^\ast Q )$ is isomorphic to $(Q \times Q^\ast) \times (Q \times Q^\ast)$, with a typical element denoted by $(\rho, \pi, \dot{\rho}, \dot{\pi})$, while $T^\ast (T^\ast Q) = (Q \times Q^\ast) \times (Q^\ast \times Q)$, with a typical element denoted by $(\rho, \pi, e_\rho, e_\pi)$.  The duality pairing between $T(T^\ast Q)$ and $T^\ast (T^\ast Q)$ is given by 
\begin{equation} \label{origdual}
	\left<  (\rho, \pi, e_\rho, e_\pi), (\rho, \pi, \dot{\rho}, \dot{\pi}) \right>
	= \int_M ( e_\rho \wedge \dot{\rho} + e_\pi \wedge \dot{\pi}) 
\end{equation}
Whenever the base point $(\rho, \pi)$ is clear from the context, we will denote $(\rho, \pi, \dot{\rho}, \dot{\pi})$ and $(\rho, \pi, e_\rho, e_\pi)$ simply by $(\dot{\rho}, \dot{\pi})$ and $(e_\rho, e_\pi)$.

\newsec{The Symmetry Group}
The set of $(k - 1)$-forms is a vector space, and hence an Abelian group, which we denote by $G$.  The group $G$ acts on $Q$ by the following additive action: for $\alpha \in G$ and $\rho \in Q$, 
\begin{equation} \label{addaction}
	\rho \cdot \alpha  = \rho + \mathbf{d} \alpha.
\end{equation}
This action lifts to $TQ$ and $T^\ast Q$ in the standard way and is given explicitly by 
\[
	 (\rho, \dot{\rho}) \cdot \alpha = (\rho + \mathbf{d} \alpha, \dot{\rho})
	\quad \text{and} \quad 
	 (\rho, \pi) \cdot \alpha = (\rho + \mathbf{d} \alpha, \pi)
\]
for $\alpha \in G$, $(\rho, \dot{\rho}) \in TQ$ and $(\rho, \pi) \in T^\ast Q$.

\newsec{The Reduced Configuration Space}
If we assume that the $k$-th de Rham cohomology of $M$ vanishes, then the quotient space $Q/G$ can be given an explicit description (see \cite{ArnoldKhesin}).  The elements of $Q/G$ are equivalence classes $[\rho]$ of $k$-forms up to exact forms, so that the exterior differential determines a well-defined map from $Q/G$ to $\mathbf{d} \Omega^k$, given by $[\rho] \mapsto \mathbf{d} \rho$.  As $H^k(M, \mathbb{R}) = 0$, this map is an isomorphism and so $Q / G = \mathbf{d} \Omega^k$.  As a result, we have that the quotient $(T^\ast Q/G)$ is isomorphic to $Q/G \times Q^\ast$, or explicitly
\[
	(T^\ast Q)/G =  \mathbf{d} \Omega^k \times \Omega^{n - k}. 
\]
For future reference, we denote the quotient map by $\pi_G : T^\ast Q \rightarrow (T^\ast Q)/G$.  It is given by 
\begin{equation} \label{quotmap}
	\pi_G (\rho, \pi) = (\mathbf{d} \rho, \pi).
\end{equation}

We will denote a typical element of $T^\ast Q/G$ by $(\bar{\rho}, \bar{\pi})$, with $\bar{\rho} \in \mathbf{d} \Omega^k$ and $\bar{\pi} \in \Omega^{n - k}$.  Elements of $T(T^\ast Q/G)$ will be denoted by $(\bar{\rho}, \bar{\pi}, \dot{\bar{\rho}}, \dot{\bar{\pi}})$, while the elements of $T^\ast(T^\ast Q/G)$ will be denoted by $(\bar{\rho}, \bar{\pi}, \bar{e}_\rho, \bar{e}_\pi)$.  For the duality pairing, we use the same sign conventions as in \eqref{origdual}, namely we put
\[ 
	\left<  (\bar{\rho}, \bar{\pi}, \bar{e}_\rho, \bar{e}_\pi), 
		(\bar{\rho}, \bar{\pi}, \dot{\bar{\rho}}, \dot{\bar{\pi}}) \right>
	= \int_M ( \bar{e}_\rho \wedge \dot{\bar{\rho}} + 
		\bar{e}_\pi \wedge \dot{\bar{\pi}}  ). 
\]
As before, whenever the base point $(\bar{\rho}, \bar{\pi})$ is clear, we will denote $(\bar{\rho}, \bar{\pi}, \dot{\bar{\rho}}, \dot{\bar{\pi}})$ simply by $(\dot{\bar{\rho}}, \dot{\bar{\pi}})$, and similarly for $(\bar{\rho}, \bar{\pi}, \bar{e}_\rho, \bar{e}_\pi)$.

\newsec{The Reduced Dirac Structure}
The cotangent bundle $T^\ast Q$ is equipped with a canonical symplectic form $\omega$ given by $\omega(\rho, \pi) ((\dot{\rho}, \dot{\pi}), (\dot{\rho}', \dot{\pi}')) 
	= \left< \dot{\pi}', \dot{\rho} \right> - 
		\left< \dot{\pi}, \dot{\rho}' \right>$ for $(\dot{\rho}, \dot{\pi})$ and $(\dot{\rho}', \dot{\pi}')$ in $T_{(\rho, \pi)} (T^\ast Q) = Q \times Q^\ast$.   The symplectic form $\omega$ induces a linear isomorphism $\sharp : T^\ast( T^\ast Q ) \rightarrow T(T^\ast Q)$ given by 
\begin{equation} \label{canonsymp}
	\sharp ( \rho, \pi, e_\rho, e_\pi) = (\rho, \pi, e_\pi, -e_\rho).
\end{equation}

Our goal is now to investigate the induced Poisson structure $[\sharp]$ on the quotient $(T^\ast Q)/G$.  This map was defined in \eqref{prescription}.  Having described the unreduced Poisson structure $\sharp$ in \eqref{canonsymp}, it now remains for us to give an explicit description of the maps $T \pi_G$ and $T^\ast \pi_G$ in the diagram.  To this end, we consider an element $(\rho, \pi) \in T^\ast Q$, and we recall that $\pi_G(\rho, \pi) = (\mathbf{d} \rho, \pi)$.  Let $T_{(\rho, \pi)} \pi_G : T_{(\rho, \pi)} (T^\ast Q) \rightarrow T_{(\mathbf{d}\rho, \pi)} (T^\ast Q/G)$ be the tangent map to $\pi_G$ at $(\rho, \pi)$ and consider the adjoint map $T^\ast_{(\rho, \pi)} \pi_G :  T^\ast_{(\mathbf{d}\rho, \pi)} (T^\ast Q/G) \rightarrow  T^\ast_{(\rho, \pi)} (T^\ast Q)$.  

\begin{lemma}  
	The tangent and cotangent maps $T_{(\rho, \pi)} \pi_G$ and $T^\ast_{(\rho, \pi)} \pi_G$ are given by 
	\begin{equation} \label{tangent}
		T_{(\rho, \pi)} \pi_G(\rho, \pi, \dot{\rho}, \dot{\pi})
			= (\mathbf{d} \rho, \pi, \mathbf{d}\dot{\rho}, \dot{\pi})
	\end{equation}
	and 
	\begin{equation} \label{cotangent}
		T^\ast_{(\rho, \pi)} \pi_G(\mathbf{d} \rho, \pi, \bar{e}_\rho, \bar{e}_\pi) =
			(\rho, \pi, (-1)^{n-k} \mathbf{d} \bar{e}_\rho, \bar{e}_\pi).
	\end{equation}
\end{lemma}

\medskip
\begin{proof}
The expression \eqref{tangent} for $T_{(\rho, \pi)} \pi_G$ is clear from the corresponding expression for $\pi_G$.  To prove \eqref{cotangent}, we let 
$(\dot{\rho}, \dot{\pi}) \in T_{(\rho, \pi)} (T^\ast Q)$ and consider 
\begin{align*}
\left< T^\ast_{(\rho, \pi)} \pi_G(\bar{e}_\rho, \bar{e}_\pi), 
	(\dot{\rho}, \dot{\pi}) \right>  & = \left< (\bar{e}_\rho, \bar{e}_\pi), 
	T_{(\rho, \pi)} \pi_G(\dot{\rho}, \dot{\pi}) \right> \\
	& = \left< (\bar{e}_\rho, \bar{e}_\pi), 
	(\mathbf{d}\dot{\rho}, \dot{\pi}) \right>.
\end{align*}
By using Stokes' theorem, this can be rewritten as
\begin{align*} 
\left< (\bar{e}_\rho, \bar{e}_\pi), 
	(\mathbf{d}\dot{\rho}, \dot{\pi}) \right>
& = \int_M ( \bar{e}_\rho \wedge \mathbf{d}\dot{\rho} + \bar{e}_\pi \wedge\dot{\pi}) \\
& = \int_M ( (-1)^{n-k}\mathbf{d}\bar{\alpha} \wedge \dot{\rho} + \bar{\beta}\wedge\dot{\pi})
\end{align*}
so that $T^\ast_{(\rho, \pi)} \pi_G(\bar{e}_\rho, \bar{e}_\pi) = ( (-1)^{n-k}\mathbf{d}\bar{e}_\rho, \bar{e}_\pi)$.
\end{proof}
\medskip

The reduced Poisson structure in \eqref{prescription} can now be obtained explicitly by composing the various constituent maps: 
\[
	[\sharp]_{(\mathbf{d} \rho, \pi)} = 
		T_{(\rho, \pi)} \pi_G \circ \sharp \circ T^\ast_{(\mathbf{d} \rho, \pi)} \pi_G
\]
for all $(\mathbf{d} \rho, \pi) \in T^\ast Q/G$.  Explicitly, we have
\begin{equation} \label{redpoisson}
	[\sharp](\bar{e}_\rho, \bar{e}_\pi) = 
		( \mathbf{d} \bar{e}_\pi, (-1)^{n-k-1} \mathbf{d} \bar{e}_\rho).
\end{equation}

\medskip
\section{Stokes-Dirac Structures}

In order to made the link between the reduced Poisson structure and the standard representation of Stokes-Dirac structures, we write \eqref{redpoisson} in matrix form:
\begin{equation} \label{matform1}
	\begin{pmatrix}
		\dot{\bar{\rho}} \\
		\dot{\bar{\pi}}
	\end{pmatrix}
	= 
	\begin{pmatrix}
		0 & \mathbf{d} \\
		(-1)^{n - k} \mathbf{d} & 0  
	\end{pmatrix}
	\begin{pmatrix}
		\bar{e}_\rho \\
		\bar{e}_\pi
	\end{pmatrix}.
\end{equation}
While the structure matrix has the same general form as the one employed in \cite{VaMa2002}, the relative signs are different.  This can be remedied by introducing new \emph{flow variables} $f_p, f_q$ and \emph{effort variables} $e_p, e_q$ defined as follows: 
\[
	e_p = \bar{e}_\rho, \quad e_q = (-1)^r \bar{e}_\pi,
	f_p = \dot{\bar{\rho}}, \quad f_q = (-1)^{n-p} \dot{\bar{\pi}}.
\]
Here, we have put $p = n - k$, $q = k + 1$, and $r = pq + 1$.  Note that $p + q = n + 1$.  With this choice of signs, \eqref{matform1} becomes 
\[
	\begin{pmatrix}
		f_p \\
		f_q 
	\end{pmatrix}
	= 
	\begin{pmatrix}
		0 & (-1)^r\mathbf{d} \\
		\mathbf{d} & 0  
	\end{pmatrix}
	\begin{pmatrix}
		e_p \\
		e_q
	\end{pmatrix}.
\]
This agrees precisely with the definition of Stokes-Dirac structures given in \cite{VaMa2002}.

\section{Examples}

We now revisit the examples of Stokes-Dirac structures given in \cite{VaMa2002}, showing in each case how the corresponding structure can be derived through reduction.

\newsec{Telegrapher's Equations}
Let $M$ be the real line with coordinate $x$.  The transmission line equations describe the propagation of currents $I(x, t)$ and voltages $V(x, t)$ through a lossless uniform wire and are given by 
\[
LI_t + V_x = 0 \quad \text{and} \quad CV_t + I_x = 0,
\]
where $L$ and $C$ are the distributed inductance and capacitance (see \cite{VaMa2002, JeVa2008} for more information).  We introduce also the integrated charge density as 
\[
	\mathcal{I}(x, t) = \int_{t_0}^t I(x, t') dt'.
\]

The lower boundary $t_0$ of the integral is arbitrary and as a result, $\mathcal{I}(x, t)$ is defined only up to an abitrary constant.  Hence, the configuration space for the transmission line equations is the space $Q := \Omega^0$ of one-forms $\mathcal{I}$, and the symmetry group acting on $Q$ is nothing but $G =\mathbb{R}$.  The quotient space $Q/G$ consists of integrated charge densities up to a constant and can be identified with $\mathbf{d} \Omega^0$ whose elements $\rho := \mathbf{d} \mathcal{I} = \mathcal{I}_x dx$ represent \emph{\textbf{charge densities}}.  

The transmission line equations fit into the framework of Stokes-Dirac structures: $Q$ and $G$ are defined above, and so $n = 1$ and $k = 0$.  As a result, the structure matrix is given by 
\[
\begin{pmatrix}
		0 & \mathbf{d} \\
		- \mathbf{d} & 0  
\end{pmatrix},
\]
which is precisely the expression obtained in \cite{JeVa2008}.  A similar approach can be used to derive the Stokes-Dirac structure for the vibrating string.

\newsec{Maxwell's Equations}
In the case of electromagnetism, we let $M$ be a three-dimensional Riemannian manifold without boundary, e.g. $M = \mathbb{R}^3$ with the Euclidian metric whose Hodge star is denoted by $\ast$.  To simplify the exposition, we restrict ourselves to electromagnetism in a vacuum and we choose appropriate units so that $\epsilon_0 = \mu_0 = 1$. 

We let the configuration space $Q$ be the space $\Omega^1$ of \emph{\textbf{vector potentials}} $A = A_i dx^i$ on $M$.  The group $G = \Omega^0$ of functions on $M$ acts on $Q$ as in \eqref{addaction}: 
\begin{equation} \label{EMaction}
f \cdot A = A + \mathbf{d} f, 
\end{equation}
and it is well known (see e.g. \cite{MarsdenRatiu}) that the quotient space $Q/G$ can be identified with the space $\mathbf{d} \Omega^1$ of \textbf{\emph{magnetic fields}} $B = \mathbf{d} A$.  
The fields on the tangent bundle $TQ$ are denoted by $(A, \dot{A})$, where $\dot{A} = -E$ with $E$ the electric field.  We will denote the fields on the cotangent bundle $T^\ast Q$ by $(A, \Pi)$, where $\Pi$ can be identified with $-D = - \ast E$ by means of the Legendre transform.


The Hamiltonian of electromagnetism (see \cite{Bossavit1998}) is given by 
\[
	\mathcal{H}(A, D) = \frac{1}{2} \int_M ( D \wedge \ast D + \mathbf{d} A \wedge \ast \mathbf{d} A).
\]
It is clear that $\mathcal{H}$ is invariant under the usual electromagnetic gauge symmetry \eqref{EMaction}, and hence it induces a reduced Hamiltonian given by the familiar expression 
\[
	\mathcal{H}'(B, D)
	=  \frac{1}{2} \int_M ( D \wedge \ast D + B \wedge \ast B).
\]
The variational derivatives of $\mathcal{H}'(B, D)$ are given by 
\[
	\frac{\delta \mathcal{H}'}{\delta D} = \ast D \quad
	\text{and} \quad 
	\frac{\delta \mathcal{H}'}{\delta B} = \ast B.
\]

As $n = \dim M = 3$ while $k = 1$, we have that the implicit Hamiltonian equations for electromagnetism are given by 
\[
	\begin{pmatrix}
		\dot{B} \\
		- \dot{D} 
	\end{pmatrix}
	= 
	\begin{pmatrix}
		0 & \mathbf{d} \\
		-\mathbf{d} & 0  
	\end{pmatrix}
	\begin{pmatrix}
		\ast B \\
		-\ast D
	\end{pmatrix},
\]
where the minus signs in front of $D$ are a reminder of the fact that $D = - \Pi$.  Written out in components, these equations are nothing but the Maxwell equations in terms of forms: $\dot{B} = - \mathbf{d} \ast D$ and $\dot{D} = \mathbf{d} \ast B$.

\medskip
\section{Stokes-Dirac Structures on Lie Algebras}

The idea of using Poisson reduction to derive Stokes-Dirac structures is not limited to the case where both the configuration space and the symmetry group are spaces of forms and the group action is as in \eqref{addaction}.  As long as the unreduced phase space has a symplectic or Poisson structure which is invariant under some group action, we can perform Poisson reduction and in the context of distributed Hamiltonian systems the result can rightfully be called again a Stokes-Dirac structure.  In this paragraph, we illustrate this idea by way of an example: the dynamics of a compressible isentropic fluid on a Riemannian manifold $M$ with metric $g$.  For this system, van der Schaft and Maschke propose the following Stokes-Dirac structure:
\begin{equation} \label{vdsSD}
	\frac{d}{dt}
	\begin{pmatrix} 
		\rho \\
		v 
	\end{pmatrix}
	= -
	\begin{pmatrix}
		\mathbf{d} e_v \\
		\mathbf{d} e_\rho + 
		\frac{1}{\ast \rho} \ast (
			(\ast \mathbf{d} v) \wedge (\ast e_v))
	\end{pmatrix}.
\end{equation}
where 
\[
	e_v = \frac{\delta \mathcal{H}}{\delta v}
		= \frac{1}{2} \left\Vert v \right\Vert^2 + 
		\frac{\partial}{\partial \tilde{\rho}} (\tilde{\rho} U(\tilde{\rho})), \quad
	e_\rho = \frac{\delta \mathcal{H}}{\delta \rho}
		= \mathbf{i}_{v^\sharp} \rho.	
\]
Here, $\ast : \Omega^i(M) \rightarrow \Omega^{3 - i}(M)$, $i = 0, \ldots, 3$, is the Hodge star associated to the metric $g$.  The fields $\rho$ and $v$ are the density and the velocity, respectively interpreted as a three-form and a one-form.  The function $\tilde{\rho}$ is $\ast \rho$ and $U(\tilde{\rho})$ represents the internal energy of the fluid.  The boundary terms have again been left out of the picture.  

Note that the structure matrix in \eqref{vdsSD} consists of the usual exterior derivatives, together with a convective term $\frac{1}{\ast \rho} \ast ((\ast \mathbf{d} v) \wedge (\ast e_v))$.  While this term is usually introduced in the Stokes-Dirac description a posteriori, we will now show that it can be derived from first principles through reduction.

\newsec{The Lie-Poisson Structure}
Let $G$ be an arbitrary Lie group with Lie algebra $\mathfrak{g}$.  We denote the dual of the Lie algebra by $\mathfrak{g}^\ast$.  It is well-known that $\mathfrak{g}^\ast$ is equipped with a natural Poisson structure, called the \emph{\textbf{Lie-Poisson structure}}, which can be defined as follows.  At each element $\mu \in \mathfrak{g}^\ast$, the Lie-Poisson structure determines a map 
\begin{equation} \label{LP}
	[\sharp]_\mu : T^\ast_\mu \mathfrak{g}^\ast \cong \mathfrak{g} \rightarrow T_\mu \mathfrak{g}^\ast \cong \mathfrak{g}^\ast, \quad 
	[\sharp]_\mu(\xi) = \mathrm{ad}^\ast_\xi \mu,
\end{equation}
where $\mathrm{ad}^\ast_\xi \mu$ is the co-adjoint action of $\xi \in \mathfrak{g}$ on $\mu \in \mathfrak{g}^\ast$, given by 
\begin{equation} \label{coad}
	\left< \mathrm{ad}^\ast_\xi \mu, \eta \right> 
		= \left< \mu, [\xi, \eta] \right>
\end{equation}
for all $\eta \in \mathfrak{g}$.  Note that the Lie-Poisson structure \eqref{LP} depends explicitly on the base point $\mu \in \mathfrak{g}^\ast$, in contrast to the standard Poisson structure \eqref{redpoisson}.

The Lie-Poisson structure can be obtained through Poisson reduction from the canonical symplectic structure on $T^\ast G$ so that this situation fits also in the framework of reduced Dirac structures.  For more information, we refer to \cite{MarsdenRatiu}.

\newsec{Compressible Isentropic Fluids}
In the case of an compressible fluid on a three-dimensional manifold $M$, the relevant group is the semi-direct product $S$ of the group $\mathrm{Diff}(M)$ of diffeomorphisms of $M$ with the space $\mathcal{F}(M)$ of functions on $M$ (see \cite{HoMaRa1998, MaRaWe1984}).  The multiplication in $S$ is given by $(\phi_1, f_1) \cdot (\phi_2, f_2) = (\phi_1 \circ \phi_2, f_2 + \phi_2^\ast f_1)$.

As a vector space, the Lie algebra $\mathfrak{s}$ of $S$ is the product $\mathfrak{X}(M) \times \mathcal{F}(M)$, where $\mathfrak{X}(M)$ is the space of all vector fields on $M$.  The bracket in $\mathfrak{s}$ is given by $[(\xi_1, f_1), (\xi_2, f_2)]
	= ( -[\xi_1, \xi_2]_M, 
		\pounds_{\xi_2} f_1 - \pounds_{\xi_1} f_2)$, 
where we have denoted the bracket of the vector fields $\xi_1, \xi_2$ on $M$ by $[\xi_1, \xi_2]_M$.  The dual $\mathfrak{s}^\ast$ can be identified with the product of the one-form densities $\Omega^1(M) \otimes \Omega^3(M)$ with the space $\Omega^3(M)$ of functions on $M$.  The duality pairing between elements $(\xi, f) \in \mathfrak{s}$ and $(\theta \otimes \rho, \rho) \in \mathfrak{s}^\ast$ is given by 
\[
	\left< (\xi, f), (\theta \otimes \rho, \rho) \right> 
	= \int_M (\theta(X) + f) \rho.
\]
Physically speaking, the field $\theta \otimes \rho$ encodes the \emph{momentum} of the fluid, while $\rho$ represents the \emph{density} and $\theta$ is the \emph{velocity}, interpreted as a one-form.

The coadjoint action \eqref{coad} is given in this case by 
\begin{equation} \label{coaddiff}
	\mathrm{ad}_{(\xi, f)}^\ast (\theta \otimes \rho, \rho) = [ (\pounds_\xi \theta + \mathrm{div}_\mu \xi + \mathbf{d} f) \otimes \rho,  \pounds_\xi \theta ].
\end{equation}

\newsec{Stokes-Dirac Structures in the Momentum Representation}
We can now introduce a Stokes-Dirac structure for compressible isentropic fluids as the graph of the Lie-Poisson map \eqref{LP}, where the co-adjoint action was computed previously.  If we denote variables on $T_{(\mathbf{m}, \rho)} \mathfrak{s}^\ast \cong \mathfrak{s}^\ast$ by $(\dot{\mathbf{m}}, \dot{\rho})$, and variables on $T^\ast_{(\mathbf{m}, \rho)} \mathfrak{s}^\ast \cong \mathfrak{s}$ by $(e_{\mathbf{m}},e_\rho)$, we have that the Stokes-Dirac structure is given by 
\begin{equation} \label{SDfluidmom}
			\begin{pmatrix}
			\dot{\mathbf{m}} \\
			\dot{\rho} 
		\end{pmatrix}
		=
		\begin{pmatrix}
		(\pounds_{e_{\mathbf{m}}^\flat} \theta + 
		\mathrm{div}_\rho e_{\mathbf{m}}^\flat + \mathbf{d} e_\rho) \otimes \rho \\
		\pounds_{e_{\mathbf{m}}^\flat} \theta
		\end{pmatrix}.
\end{equation}
Here we have used the metric on $M$ to identify the one-form $e_{\mathbf{m}}$ with a vector field $e_{\mathbf{m}}^\flat$.  As this expression is written primarily in terms of the fluid momentum $\mathbf{m}$, we refer to it as being in the momentum representation.  To make the link with \eqref{vdsSD}, we need to rewrite it in terms of the velocity and the density.

\newsec{The Velocity Representation}
Let $(\mathbf{m}, \rho) \in \mathfrak{s}^\ast$.  We define the \emph{\textbf{velocity}} of the fluid as the one-form $\theta$ defined implicitly by $\mathbf{m} = \theta \otimes \rho$.  To make the link with velocity as a vector field, use the metric to define $v := \theta^\flat$.

We let $V$ be the space $\Omega^1(M) \times \Omega^3(M)$ of velocity-density pairs $(\theta, \rho)$.  Its dual $V^\ast$ is the space $\Omega^2(M) \times \Omega^0(M)$ whose elements are denoted by $(e_\theta, e_\rho)$.  The duality pairing is given by 
\[
	\left< (e_\theta, e_\rho), (\theta, \rho) \right> 
		= \int_M (e_\theta \wedge \theta + e_\rho \rho).
\]

To express the Stokes-Dirac structure \eqref{SDfluidmom} in terms of the velocity-density variables $(\theta, \rho) \in V$, we introduce the map $\Phi: \mathfrak{g}^\ast \rightarrow V$ relating the momentum representation with the velocity representation: 
\[
	\Phi : (\mathbf{m}, \rho) 
	\mapsto (\theta, \rho), \quad 
		\text{where $\mathbf{m} = \theta \otimes \rho$}.
\]

The structure of the tangent map $T \Phi$ and the cotangent map $T^\ast \Phi$ is elucidated in the next two lemmas. 

\begin{lemma} \label{lemma:tangent} 
Let $(\mathbf{m}, \rho) \in \mathfrak{s}^\ast$ such that $\mathbf{m} = \theta \otimes \rho$.  The tangent map $T_{(\mathbf{m}, \rho)} \Phi: \mathfrak{s}^\ast \rightarrow V$ can be written as 
$T_{(\mathbf{m}, \rho)} \Phi(\dot{\mathbf{m}}, \dot{\rho} ) 
	= (\dot{\theta}, \dot{\rho})$, 
where the relation between $\dot{\mathbf{m}}$ and $\dot{\theta}$ is given by $\dot{\mathbf{m}} = \theta \otimes \dot{\rho} + \dot{\theta} \otimes \rho$ if $\mathbf{m} = \theta \otimes \rho$.
\end{lemma}

The proof of this lemma is nothing but the observation that the rate of change $\dot{\mathbf{m}}$ of the momentum density is influenced both by the rate of change of velocity $\dot{\theta} \otimes \rho$ as well as by the change in density $\theta \otimes \dot{\rho}$.  The structure of the dual map $T_{(\mathbf{m}, \rho)}^\ast \Phi$ is clarified in the following lemma, the proof of which is given in the appendix.  

\begin{lemma} \label{lemma:dual} 
Consider an element $(\mathbf{m}, \rho)$ of $\mathfrak{s}^\ast$ such that $\mathbf{m} = \theta \otimes \rho$ and an element $(e_\theta, e_\rho)$ of $V^\ast$.  The dual mapping $T_{(\mathbf{m}, \rho)}^\ast \Phi : V^\ast \rightarrow \mathfrak{s}$ is then given by 
\[
	T_{(\mathbf{m}, \rho)}^\ast \Phi (e_\theta, e_\rho) = \left( \frac{(\ast e_\theta)^\sharp}{\ast \rho}, e_\rho - \frac{\ast( e_\theta \wedge \theta )}{\ast \rho} \right). 
\]
\end{lemma}  

\newsec{Stokes-Dirac Structures in the Velocity Representation}
Let $(\theta, \rho)$ be an element of $V$ and put $\mathbf{m} = \theta \otimes \rho$.  The reduced Lie-Poisson structure 
$[\sharp]_{(\theta, \rho)}: V^\ast \rightarrow V$ in the velocity representation is given as the composition of three maps 
\[
[\sharp]_{(\theta, \rho)} = 
T_{(\mathbf{m}, \rho)} \Phi \circ
\circ [\sharp]_{(\mathbf{m}, \rho)} \circ
T_{(\mathbf{m}, \rho)}^\ast \Phi.
\]

\begin{theorem} \label{SDvel}
The reduced Lie-Poisson structure in the velocity representation is given by 
\[
	[\sharp]_{(\theta, \rho)}(e_\theta, e_\rho) 
	= \left( \mathbf{d} e_\rho + 
	\frac{1}{\ast \rho}
	\mathbf{i}_{(\ast e_\theta)^\sharp} \mathbf{d} \theta, \mathbf{d} e_\theta
\right).
\]
\end{theorem}

\begin{proof}
Consider an element $(e_\theta, e_\rho)$ of $V^\ast$.   
For the sake of conciseness, given a two-form $e_\theta$ we define a vector field $X_{e_\theta}$ given by 
\begin{equation} \label{vector}
	X_{e_\theta} = \frac{(\ast e_\theta)^\sharp}{\ast \rho}.
\end{equation}
The value of $[\sharp]_{(\theta, \rho)}(e_\theta, e_\rho)$ is now given by the following diagram.  
\[
\xymatrix{
(e_\theta, e_\rho) \ar[d]^{T^\ast \Phi} \\
\left( X_{e_\theta}, e_\rho - \frac{\ast( e_\theta \wedge \theta )}{\ast \rho} \right)
\ar[d]^{[\sharp]} \\
\left(  \left(\pounds_{X_{e_\theta}} \theta + \mathrm{div}_\rho (X_{e_\theta}) \theta + \mathbf{d} e_\rho - \mathbf{d} \left(
	\frac{\ast( e_\theta \wedge \theta )}{\ast \rho} \right)
	\right) \otimes \rho , \pounds_{X_{e_\theta}} \rho \right) \ar[d]^{T\Phi} \\
\left( \mathbf{d} e_\rho + 
	\mathbf{i}_{X_{e_\theta}} \mathbf{d} \theta, \mathbf{d} e_\theta
\right) 	
}\]
Here, the first two maps have been given in lemma~\ref{lemma:dual} and in \eqref{SDfluidmom}.  The third expression can be simplified by using the fact that
\[
	\mathbf{i}_{X_{e_\theta}} \theta   =
	\frac{1}{\ast \rho} g( \theta, \ast e_\theta)  =
	 \frac{\ast(e_\theta \wedge \theta)}{\ast \rho}  
\]
so that 
\begin{align*}
	\pounds_{X_{e_\theta}} \theta & = 
		\mathbf{i}_{X_{e_\theta}} \mathbf{d}  \theta +
		\mathbf{d}  \mathbf{i}_{X_{e_\theta}} \theta \\
		& = \mathbf{i}_{X_{e_\theta}} \mathbf{d}  \theta 
		+ \mathbf{d} \left( 
		\frac{\ast(e_\theta \wedge \theta)}%
			{\ast \rho}\right). 
\end{align*}
The third expression hence simplifies to 
\[
(\dot{\mathbf{m}}, \dot{\rho}) :=
	\left( (\mathbf{i}_{X_{e_\theta}} \mathbf{d}  \theta
	+ \mathrm{div}_\rho (X_{e_\theta}) \theta + \mathbf{d} e_\rho)\otimes \rho, 
	\pounds_{X_{e_\theta}} \rho
	\right).
\]
The effect of the map $T\Phi$ is to cancel the divergence in the above expression.  Since $\dot{\rho} = 
		\pounds_{X_{e_\theta}} \rho = \mathrm{div}_\rho (X_{e_\theta}) \rho$,
we have that $\dot{\theta}$ is defined by 
\begin{align*}
	\dot{\mathbf{m}} & = \dot{\theta} \otimes \rho + 
		\theta \otimes \dot{\rho} \\
		& = 
	(\dot{\theta} + \mathrm{div}_\rho(X_{e_\theta}) 
		\theta) \otimes \rho
\end{align*}
and therefore $\dot{\theta} = \mathbf{i}_{X_{e_\theta}} \mathbf{d}  \theta + \mathbf{d} e_\rho$.  Noting furthermore that, since $\rho$ has maximal degree, 
\[
\pounds_{X_{e_\theta}} \rho = 
	\mathbf{d} \mathbf{i}_{X_{e_\theta}} \rho
	= \mathbf{d} e_\theta,
\]
we have that $T\Phi(\dot{\mathbf{m}}, \dot{\rho})$ is given by 
\[
	T\Phi(\dot{\mathbf{m}}, \dot{\rho})
	= (\dot{\theta}, \dot{\rho}) 
	= (\mathbf{i}_{X_{e_\theta}} \mathbf{d}  \theta + \mathbf{d} e_\rho, \mathbf{d} e_\theta).
\]
This concludes the proof.
\end{proof}

\medskip
Finally, the Stokes-Dirac structure described in theorem~\ref{SDvel} can be made to agree with \eqref{SDfluidmom} by observing that the convective term can be rewritten as (see \cite{VaMa2002} for a proof)
\begin{equation} \label{convective}
\frac{1}{\ast \rho}
	\mathbf{i}_{(\ast e_\theta)^\sharp} \mathbf{d} \theta
	= \frac{1}{\ast \rho} \ast(\ast e_\theta \wedge \ast\mathbf{d} \theta).
\end{equation}
%

In this way, we recover precisely the Stokes-Dirac structure \eqref{vdsSD}.  Whereas the convective term \eqref{convective} was introduced in \cite{VaMa2002} on an ad-hoc basis in order to reproduce Euler's equations, it appears here in a natural way through Poisson reduction.

\medskip
\section{Conclusions and Outlook}

In this paper we have studied Stokes-Dirac structures from a geometric point of view.  We have shown that most common examples of Stokes-Dirac structures arise through symmetry reduction of a canonical Dirac structure on an infinite-dimensional phase space. 

We now sketch some directions for future research.  Stokes-Dirac structures were originally conceived in \cite{VaMa2002} to deal with systems that are controlled through the boundary.  It would be of considerable interest to extend the reduction methods of this paper to boundary control systems.  In the case of Stokes-Dirac structures on the space $\Omega^k$ of $k$-forms, boundary controls can be incorporated in the reduction method by a suitable identification of the dual space $(\Omega^k)^\ast$, allowing forms with support concentrated on the boundary.



Secondly, the Stokes-Dirac description starts from field theories in which time is treated on a different footing from the spatial variables.  While covariant Stokes-Dirac structures have been addressed in \cite{VaMa2002}, more remains to be done.  It is well-known (see for instance \cite{gimmsyII}) that the canonical symplectic form on the space of fields arises from a covariant multisymplectic form once a splitting of space and time is chosen.  It is therefore likely that covariant Stokes-Dirac structures can be defined directly in terms of multisymplectic forms.  This, and the link with the multi-Dirac structures of \cite{VaYoMa2010}, will be explored in a forthcoming publication.

\appendix[Proof of Lemma~\ref{lemma:dual}]

Throughout this paragraph, we use the notations of lemma~\ref{lemma:dual}.  The value of $T_{(\mathbf{m}, \rho)}^\ast \Phi (e_\theta, e_\rho)$ is defined by   
\begin{align}
\left< T_{(\mathbf{m}, \rho)}^\ast \Phi (e_\theta, e_\rho), 
(\dot{\mathbf{m}}, \dot{\rho}) \right> & =
\left< (e_\theta, e_\rho), T_{(\mathbf{m}, \rho)} \Phi 
(\dot{\mathbf{m}}, \dot{\rho}) \right> \nonumber \\
& = \left< (e_\theta, e_\rho), (\dot{\theta}, \dot{\rho}) \right> \nonumber \\
& = \int_M (\dot{\rho} e_\rho + \dot{\theta} \wedge e_\theta).
\label{integral}
\end{align}
We introduce $\tilde{\rho} := \ast \rho$, $\dot{\tilde{\rho}} := \ast \dot{\rho}$ and we let $\dot{m}$ be the one-form defined by $\dot{\mathbf{m}} = \dot{m} \otimes dV$, where $dV$ is the Riemannian volume form.  The one-form $\dot{\theta}$ can then be expressed as $\dot{\theta} = (\dot{m} - \dot{\tilde{\rho}}\theta)/\tilde{\rho}$. The second term in the integrand of \eqref{integral} can therefore be written as
\begin{equation} \label{diff}
\dot{\theta} \wedge e_\theta  = 
\frac{1}{\tilde{\rho}}\left( 
\dot{m} \wedge e_\theta
- \dot{\tilde{\rho}} \theta \wedge e_\theta 
\right).
\end{equation}
We recall the definition \eqref{vector} of the vector field $X_{e_\theta}$ associated to the two-form $e_\theta$ and note that $\mathbf{i}_{X_{e_\theta}} dV = e_\theta/\tilde{\rho}$.  We now address the first term in \eqref{diff}.  Since $dV$ has maximal degree, $\dot{m} \wedge dV = 0$ and therefore 
\[
	0 = \mathbf{i}_{X_{e_\theta}} (\dot{m} \wedge dV) 
	= (\mathbf{i}_{X_{e_\theta}} \dot{m}) dV
	- \dot{m}  \wedge \mathbf{i}_{X_{e_\theta}} dV.
\]
As a result, 
\begin{equation} \label{rewrite}
(\mathbf{i}_{X_{e_\theta}} \dot{m}) dV
	= \dot{m}  \wedge \mathbf{i}_{X_{e_\theta}} dV
	= \tilde{\rho}^{-1} \dot{m} \wedge e_\theta.
\end{equation}

Secondly, we can trivially rewrite $\theta \wedge e_\theta = \ast(\theta \wedge e_\theta) dV$, and by substituting this and \eqref{rewrite} into \eqref{diff}, we obtain 
\[
	\dot{\theta} \wedge e_\theta 
	= 
	(\mathbf{i}_{X_{e_\theta}} \dot{m}) \, dV 
	- \frac{\ast (\theta \wedge e_\theta)}{\tilde{\rho}} 
		\dot{\rho}.
\]
The integral \eqref{integral} then becomes
\begin{multline*}
\left< T_{(\mathbf{m}, \rho)}^\ast \Phi (e_\theta, e_\rho), 
(\dot{\mathbf{m}}, \dot{\rho}) \right> \\
 = \int_M \left[
	\left(e_\rho - 
	\frac{\ast (\theta \wedge e_\theta)}{\tilde{\rho}} 
		\right) \dot{\rho} +
	(\mathbf{i}_{X_{e_\theta}} \dot{m}) 
	\right] \, dV,
\end{multline*}
which concludes the proof of lemma~\ref{lemma:dual}.


\end{document}